\documentclass[12pt, reqno]{amsart}

\usepackage{amssymb,amsmath,latexsym}
\usepackage[varg]{pxfonts}
\usepackage{amsthm, amscd, amsfonts, amssymb, graphicx, color}
\usepackage[bookmarksnumbered, colorlinks, plainpages]{hyperref}
\usepackage{marvosym}
\usepackage{paralist}
\usepackage{color}
\theoremstyle{plain}
\newtheorem{lem}{\textbf{Lemma}}[section]
\newtheorem{thm}[lem]{\textbf{Theorem}}

\newtheorem{de}[lem]{\textbf{Definition}}
\newtheorem{op}{\textbf{Open problem}}
\newtheorem{ex}[lem]{\textbf{Example}}
\newtheorem{pr}[lem]{\textbf{Proposition}}
\theoremstyle{definition}

\theoremstyle{Remark}
\newtheorem{rem}[lem]{\textbf{Remark}}

\begin{document}

\noindent {}   \\[0.50in]


\title[A Strong Convergence Theorem for Bregman Demimetric Mappings]{A Strong Convergence Theorem for a finite family of Bregman Demimetric Mappings in a Banach Space under a New Shrinking projection Method}

\author[  Orouji, Soori, O'Regan, Agarwal]{Bijan Orouji$^{1}$,  Ebrahim Soori$^{2,*}$,  Donal O'Regan$^{3}$, Ravi P. Agarwal${^{4}}$}
\thanks{$^{*}$ Corresponding author}

\address{$^{1,2}$Department of Mathematics, Lorestan University,  P.O. Box 465, Khoramabad, Lorestan, Iran\\
$^{3}$School of Mathematics, Statistics, National University of Ireland, Galway, Ireland\\
$^{4}$Department of Mathematics, Texas A $\&$ M University Kingsville, Kingsville, USA,}

\email{bijanorouji@yahoo.com(B. Orouji); sori.e@lu.ac.ir(E. Soori);
donal.oregan@nuigalway.ie(D. O'Regan); agarwal@tamuk.edu(R. P. Agarwal)}



\maketitle
\hrule width \hsize \kern 1mm

\begin{abstract}
In this paper, using a new shrinking projection method and new generalized $k$-demimetric mappings, we consider the strong convergence for finding a common point of the sets of zero points of maximal monotone mappings, common fixed points of a finite family of Bregman $k$-demimetric mappings and common zero points of a finite family of Bregman inverse strongly monotone mappings in a reflexive Banach space. To the best of our knowledge such a theorem for Bregman $k$-demimetric mapping is the first of its kind in a Banach space.


\end{abstract}
\maketitle
\vspace{0.1in}
\hrule width \hsize \kern 1mm

\section{\textbf{Introduction}}
Let $H$ be a Hilbert space and let $C$ be a nonempty, closed and convex subset of H. Let $T:C\rightarrow H$ be a mapping. Then we denote by $F(T)$ the set of fixed points of $T$. For a real number $t$ with $0\leq t\leq1$, a mapping $U:C\rightarrow H$ is said to be a $t$-strict pseudo-contraction \cite{p9} if
\begin{equation*}
\|Ux-Uy\|^2\leq\|x-y\|^2+t\|x-Ux-(y-Uy)\|^2,
\end{equation*}
for all $x,y\in C$. In particular, if $t=0$, then $U$ is nonexpansive, i.e.,
\begin{equation*}
\|Ux-Uy\|\leq \|x-y\|,\;\forall x,y\in C.
\end{equation*}
If $U$ is a $t$-strict pseudo-contraction and $F(U)\neq\emptyset$, then we get that, for $x\in C$ and $p\in F(U)$,
\begin{equation*}
\|Ux-p\|^2\leq\|x-p\|^2+t\|x-Ux\|^2.
\end{equation*}
From this inequality, we get that
\begin{equation*}
\|Ux-x\|^2+\|x-p\|^2+2\langle Ux-x,x-p\rangle\leq \|x-p\|^2+t\|x-Ux\|^2.
\end{equation*}
Then we get that
\begin{equation}\label{edsnjh}
2\langle x-Ux,x-p\rangle\geq(1-t)\|x-Ux\|^2.
\end{equation}
A mapping $U:C\rightarrow H$ is said generalized hybrid \cite{p17} if there exist real numbers $\alpha,\beta$ such that
\begin{equation*}
\alpha\|Ux-Uy\|^2+(1-\alpha)\|x-Uy\|^2\leq\beta\|Ux-y\|^2+(1-\beta)\|x-y\|^2,
\end{equation*}
for all $x,y\in C$. Such a mapping $U$ is said to be $(\alpha,\beta)$-generalized hybrid. The class of generalized hybrid mappings covers several well-known mapping. A $(1,0)$-generalized hybrid mapping is nonexpansive. For $\alpha=2$ and $\beta=1$, it is nonspreading \cite{p18,p19}, i.e.,
\begin{equation*}
2\|Ux-Uy\|^2\leq\|Ux-y\|^2+\|Uy-x\|^2,\quad\forall x,y\in C.
\end{equation*}
For $\alpha=\frac{3}{2}$ and $\beta=\frac{1}{2}$, it is also hybrid \cite{p27}, i.e.,
\begin{equation*}
3\|Ux-Uy\|^2\leq\|x-y\|^2+\|Ux-y\|^2+\|Uy-x\|^2,\quad\forall x,y\in C.
\end{equation*}
In general, nonspreading mappings and hybrid mappings are not continuous: see \cite{p15}. If $U$ is a generalized hybrid and $F(U)\neq\emptyset$, then we get that, for $x\in C$ and $p\in F(U)$,
\begin{equation*}
\alpha\|p-Ux\|^2+(1-\alpha)\|p-Ux\|^2\leq\beta\|p-x\|^2+(1-\beta)\|p-x\|^2
\end{equation*}
and hence $\|Ux-p\|^2\leq \|x-p\|^2$. From this, we have that
\begin{equation}\label{xcrjv}
2\langle x-p,x-Ux\rangle\geq\|x-Ux\|^2.
\end{equation}
Let $E$ be a smooth Banach space and let $G$ be a maximal monotone mapping with $G^{-1}0\neq\emptyset$. Then, for the metric resolvent $J_\lambda$ of $G$ for a positive number $\lambda>0$, we obtain from \cite{p1,p28} that, for $x\in E$ and $p\in G^{-1}0=F(J_\lambda)$,
\begin{equation*}
\langle J_\lambda x-p,J(x-J_\lambda x)\rangle\geq0.
\end{equation*}
Then we get
\begin{equation*}
\langle J_\lambda x-x+x-p,J(x-J_\lambda x)\rangle\geq0.
\end{equation*}
and hence
\begin{equation}\label{eiuyn}
\langle x-p,J(x-J_\lambda x)\rangle\geq\|x-J_\lambda x\|^2,
\end{equation}
where $J$ is the duality mapping on $E$. Motivated by \eqref{edsnjh}, \eqref{xcrjv} and \eqref{eiuyn}, Takahashi \cite{p29} introduced a nonlinear mapping in a Banach space as follows: Let $C$ be a nonempty, closed and convex subset of a smooth Banach space $E$ and let $\eta$ be a real number with $\eta\in (-\infty,1)$. A mapping $U:C\rightarrow E$ with $F(U)\neq\emptyset$ is said to be $\eta$-demimetric if, for $x\in C$ and $p\in F(U)$,
\begin{equation*}
2\langle x-p,J(x-Ux)\rangle\geq(1-\eta)\|x-Ux\|^2.
\end{equation*}
According to this definition, we have that a $t$-strict pseudo-contraction $U$ with $F(U)\neq\emptyset$ is $t$-demimetric, an $(\alpha,\beta)$-generalized hybrid mapping $U$ with $F(U)\neq\emptyset$ is $0$-demimetric and the metric resolvent $J_\lambda$ with $G^{-1}0\neq\emptyset$ is $(-1)$-demimetric.

On the other hand, in 1967, Bregman \cite{p7} discovered an effective technique using the so-called Bregman distance function $D_f$ in the process of designing and analysing feasibility and optimization algorithms. This led to a growing area of research in which Bregman's technique is applied in various ways in order to design and analyse iterative algorithms for solving  feasibility problems, equilibrium problems, fixed point problems for nonlinear mappings, and so on (see, e.g. \cite{p21,p25} and the references therein).

In 2010, Reich and Sabach \cite{p21}  using the Bregman distance function $D_f$ introduced the concept of Bregman strongly nonexpansive mappings and studied the convergence of two iterative algorithms for finding common fixed points of finitely many Bregman strongly nonexpansive operators in reflexive Banach spaces.

In this paper, motivated by Takahashi \cite{p31q}, we generalize $k$-demimetric mappings by the  Bregman distance and using a new shrinking projection method, we deal with the strong convergence for finding a common point of the sets of zero points of a maximal monotone mapping, common fixed points of a finite family of Bregman $k$-demimetric mapping and common zero points of finite family of Bregman inverse strongly monotone mapping in a reflexive Banach space.
\section{\textbf{Preliminaries}}
Let $E$ be a reflexive real Banach space and $C$ be a nonempty closed and convex subset of $E$. Throughout this paper, the dual space of $E$ is denoted by $E^*$. The norm and duality pairing between $E$ and $E^*$ are respectively denoted by $\|.\|$ and $\langle.,.\rangle$. Let $\{x_n\}_{n\in \mathbb{N}}$ be a sequence in $E$, and we denote the strong convergence of $\{x_n\}_{n\in \mathbb{N}}$ to $x\in E$ as $n\rightarrow\infty$ by $x_n\rightarrow x$ and the weak convergence by $x_n\rightharpoonup x$.\\
Throughout this paper, $f:E\rightarrow(-\infty,+\infty]$ is a proper, lower semicontinuous and convex function. We denote by $dom f:=\{x\in E;\;f(x)<\infty\}$, the domain of $f$.  The function f is said to be strongly coercive if $\displaystyle\lim_{\|x\|\rightarrow\infty}\frac{f(x)}{\|x\|}=+\infty$. Let $x\in$ \textit{int\,dom}$f$, and the subdifferential of $f$ at $x$ is the convex mapping set $\partial f:E\rightarrow 2^{E^*}$ defined by
\begin{equation*}
\partial f(x)=\{\xi\in E^*:\;f(x)+\langle y-x,\xi\rangle\leq f(y),\;\forall y\in E\},\quad \forall x\in E,
\end{equation*}
 and $f^*:E^*\rightarrow (-\infty,+\infty]$ is the Fenchel conjugate of $f$ defined by
\begin{align*}
  f^*(\xi)=sup\{\langle \xi,x\rangle-f(x):x\in E\}.
\end{align*}
It is well known that $\xi \in \partial f(x)$ is equivalent to
\begin{align*}
f(x)+ f^*(\xi)=\langle x, \xi\rangle.
\end{align*}
For any $x\in$ \textit{int\,dom}$f$ and $y\in E$, we denote by $f^\circ(x,y)$ the right-hand dervative of $f$ at $x$ in the direction $y$, that is,
\begin{equation}\label{solkf}
f^\circ(x,y):=\displaystyle\lim_{t\rightarrow0^+}\frac{f(x+ty)-f(x)}{t}.
\end{equation}
The function $f$ is called G$\hat{\text{a}}$teaux differentiable at $x$, if the limit in \eqref{solkf}
 exists for any $y\in E$. In this case, the gradient of $f$ at $x$ is the linear function $\nabla f$ which is defined by $\langle y,\nabla f(x)\rangle:=f^\circ(x,y)$ for any $y\in E$. The function $f$ is said to be G$\hat{\text{a}}$teaux differentiable if it is G$\hat{\text{a}}$teaux differentiable at each $x\in int\,dom f$. The function $f$ is said to be Fr$\acute{\text{e}}$chet differentiable at $x$, if the limit in \eqref{solkf} is attained uniformly in $\|y\|=1$, for any $y\in E$. Finally, $f$ is said to be uniformly Fr$\acute{\text{e}}$chet differentiable on a subset $C$ of $E$, if the limit in \eqref{solkf} is attained uniformly for $x\in C$ and $\|y\|=1$.
\begin{lem}\cite{p21}\label{zakhv}
  If $f:E\rightarrow \mathbb{R}$ is uniformly $Fr\acute{e}chet$ differentiable and bounded on bounded subsets of $E$, then $f$ is uniformly continuous on bounded subsets of $E$ and $\nabla f$ is uniformly continuous on bounded subsets of $E$ from the strong topology of $E$ to the strong topology of $E^*$.
\end{lem}
\begin{pr}\cite{p31}
 Let $f:E\rightarrow \mathbb{R}$ be a convex function which is bounded on bounded subsets of $E$. Then the following assertions are equivalent:
 \begin{itemize}
  \item [{\rm(i)}] $f$ is strongly coercive and uniformly convex on bounded subsets of $E$.
  \item [{\rm(ii)}] $f^*$ is $Fr\acute{e}chet$ differentiable and $\triangledown f^*$ is uniformly norm-to-norm continuous on bounded subsets of $domf^*=E^*$.
  \end{itemize}
\end{pr}
\begin{de}
The function $f$ is said to be 'Legendre' if it satisfies the following two conditions:
\begin{itemize}
  \item [{\rm(L1)}] $int\,dom f\neq \emptyset$ and $\partial f$ is single-valued on its domain.
  \item [{\rm(L2)}] $int\,dom f^*\neq \emptyset$ and $\partial f^*$ is single-valued on its domain.
  \end{itemize}
\end{de}
Because here the space $E$ is assumed to be reflexive, we always have $(\partial f)^{-1}=\partial f^*$. \cite[p. 83]{p6}. This fact, when combined with the conditions (L1) and (L2), implies the following equalities:
\begin{align*}
\nabla f&=(\nabla f^*)^{-1},\\
ran \nabla f&=dom \nabla f^*=int\,dom f^*,\\
ran \nabla f^*&=dom \nabla f=int\,dom f.
\end{align*}
In addition, the conditions  (L1) and (L2), in conjunction with Theorem 5.4 of \cite{p2}, imply that the functions $f$ and $f^*$ are strictly convex on the interior of their respective domains and $f$ is Legendre if and only if $f^*$ is Legendre.

One important and interesting Legendre function is $\frac{1}{p}\|.\|^p,\;p\in (1,2]$. When $E$ is a uniformly convex and $p$-uniformly smooth Banach space with $p\in (1,2]$, the generalized duality mapping $J_p:E\rightarrow2^{E^*}$ is defined by\\
\begin{equation*}
J_p(x)=\{j_p(x)\in E^*:\langle j_p(x),x\rangle=\|x\|.\|j_p(x)\|\;,\;\|j_p(x)\|=\|x\|^{p-1}\}.
\end{equation*}
In this case, the gradient $\nabla f$ of $f$ is coincide with the generalized duality mapping $J_p$ of $E$, $\nabla f=J_p,\;p\in(1,2]$.
Several interesting examples of Legendre functions are presented in \cite{p2,p3,p4}.

From now on, we always assume that the convex function $f:E\rightarrow (0,+\infty]$ is Legendre.
\begin{de}\cite{p14}
Let $f:E\rightarrow(-\infty,+\infty]$ be a convex and $G\hat{a}teaux$ differentiable function.The bifunction $D_f:domf\times int\;domf\rightarrow [0,+\infty)$ defined by
\begin{align*}
D_f(y,x):=f(y)-f(x)-\langle \nabla f(x),y-x\rangle.
\end{align*}
is called the Bregman distance with respect to f.
\end{de}
It should be noted that $D_f$ is not a distance in the usual sense of the term. Clearly, $D_f(x,x)=0$, but $D_f(y,x)=0$ may not imply $x=y$. In our case, when $f$ is Legendre this indeed holds \cite[Theorem 7.3(vi), p.642]{p2}. In general, $D_f$  satisfies the three point identity
\begin{equation}\label{awpon}
D_f(x,y)+D_f(y,z)-D_f(x,z)=\langle x-y, \nabla f(z)-\nabla f(y)\rangle,
\end{equation}
and the four point identity
\begin{equation*}
D_f(x,y)+D_f(\omega,z)-D_f(x,z)-D_f(\omega,y)=\langle x-\omega,\nabla f(z)-\nabla f(y)\rangle,
\end{equation*}
for any $x,\omega \in domf$ and $y,z\in int\,dom f$. Over the last 30 years, Bregman distances have been studied by many researchers (see \cite{p2,p5,p10,p11}).

Let $f:E\rightarrow(-\infty,+\infty]$ be a convex function on E which is G$\hat{\text{a}}$teaux differentiable on \textit{int dom}$f$. The function $f$ is said to be totally convex at a point $x\in int\,dom\,f$ if its modulus of total convexity at $x,\;v_f(x,.):[0,+\infty)\rightarrow [0,+\infty]$, defined by
\begin{align*}
 v_f(x,t)=inf \{D_f(y,x):y\in dom f,\|y-x\|=t\},
\end{align*}
is positive whenever $t>0$. The function $f$ is said to be totally convex when it is totally convex at every point of \textit{int dom}$f$. The function f is said to be totally convex on bounded sets, if for any nonempty bounded set $B\subseteq E$ , the modulus of total convexity of $f$ on $B$, $v_f(B,t)$ is positive for any $t>0$, where $v_f(B,.):[0,+\infty)\rightarrow [0,+\infty]$  is defined by
\begin{align*}
  v_f(B,t)=inf\{v_f(x,t):x\in B\cap int\,dom\,f\}.
\end{align*}
We remark in passing that $f$ is totally convex on bounded sets if and only if $f$ is uniformly convex on bounded sets; (see \cite{p11.1,p12}).
\begin{pr}\cite{p11.1}
 Let $f:E\rightarrow(-\infty,+\infty]$ be a convex function that it's domain contains at least two points. If $f$ is lower semicontinuous, then $f$ is totally convex on bounded sets if and only if $f$ is uniformly convex on bounded sets.
\end{pr}
\begin{lem}\cite{p21}
If $x\in int\;dom f$, then the following statements are equivalent:\\
(i) The function f is totally convex at x.\\
(ii) For any sequence $\{y_n\}\subset\;dom f$,
\begin{equation*}
\displaystyle\lim_{n\rightarrow +\infty}D_f(y_n,x)=0\Rightarrow \displaystyle\lim_{n\rightarrow +\infty}\|y_n-x\|=0.
\end{equation*}
\end{lem}
Recall that the function $f$ is called sequentially consistent \cite{p14.5}, if for any two sequences $\{x_n\}_{n\in \mathbb{N}}$ and $\{y_n\}_{n\in \mathbb{N}}$ in $E$ such that $\{x_n\}_{n\in \mathbb{N}}$ is bounded, then
\begin{equation*}
\displaystyle\lim_{n\rightarrow +\infty}D_f(y_n,x_n)=0\Rightarrow \displaystyle\lim_{n\rightarrow +\infty}\|y_n-x_n\|=0.
\end{equation*}
\begin{lem}\label{qpocv} \cite{p11}
If dom f contains at least two points, then the function f is totally convex on bounded sets if and only if the function f is sequentially consistent.
\end{lem}
\begin{lem}\cite{p22}\label{qwpo}
Let $f:E\rightarrow \mathbb{R}$ be a $G\hat{a}teaux$ differentiable and totally convex function. If $x_1\in E$ and the sequence $\{D_f(x_n,x_1)\}$ is bounded, then the sequence $\{x_n\}$ is also bounded.
\end{lem}
\begin{lem}\cite{p24}\label{fdyuq}
Let $f:E\rightarrow \mathbb{R}$ be a Legendre function such that $\nabla f^*$ is bounded on bounded subsets of $int\,dom f^*$. Let $x_1\in E$ and if $\{D_f(x_1,x_n)\}$ is bounded, then the sequence $\{x_n\}$ is bounded too.
\end{lem}
Recall that the Bregman projection \cite{p8} with respect to $f$ of $x\in int\,dom f$ onto a nonempty, closed and convex set $C\subseteq int\,dom f$ is the unique vector $proj_C^f(x)\in C$ satisfying
\begin{equation*}
D_f\big(proj_C^f(x),x\big)=inf\{D_f(y,x):y\in C\}.
\end{equation*}
Similar to the metric projection in Hilbert spaces, the Bregman projection with respect to totally convex and G$\hat{\text{a}}$teaux differentiable functions has a variational characterization \cite[corollary 4.4, p. 23]{p12}.
\begin{lem}\label{rtmnw}\cite{p13}
Suppose that f is $G\hat{a}teaux$ differentiable and totally convex on $int\,dom f$. Let $x\in int\,dom f$ and $C\subseteq int\,dom f$ be a nonempty, closed and convex set. Then the following Bregamn projection conditions are equivalent:\\
(i) $z_0=proj_C^f(x)$.\\
(ii) $z=z_0$ is the unique solution of the following variational inequality:
\begin{equation*}
\langle z-y,\nabla f(x)-\nabla f(z)\rangle \geq0,\quad \forall y\in C.
\end{equation*}
(iii) $z=z_0$ is the unique solution of the following variational inequality:
\begin{equation*}
D_f(y,z)+D_f(z,x)\leq D_f(y,x),\quad \forall y\in C.
\end{equation*}
\end{lem}
Let $E$ be a real Banach space and $C$ be a nonempty subset of $E$. An element $p\in C$ is called a fixed point of a single-valued mapping $T:C\rightarrow C$, if $p=Tp$. The set of fixed points of $T$ is denoted by $F(T)$.

A point $x\in C$ is called an asymptotic fixed point of $T$ if $C$ contains a sequence $\{x_n\}$ which converges weakly to $x$ and $\displaystyle\lim_{n\rightarrow +\infty}\|x_n-Tx_n\|=0$. We denote the asymptotic fixed points of $T$ by $\tilde{F}(T)$.

Let $C$ be a nonempty, closed and convex subset of \textit{int\,dom} $f$ and $T:C\rightarrow C$ be a mapping. Now $T$ is said to be Bregman quasi-nonexpansive, if $F(T)\neq \emptyset$ and
\begin{equation*}
D_f(p,Tx)\leq D_f(p,x),\quad \forall x\in C,p\in F(T).
\end{equation*}

Let $C$ be a nonempty, closed and convex subset of \textit{int dom} $f$. An operator $T:C\rightarrow int\,dom f$ is said to be Bregman strongly nonexpansive with respect to a nonempty $\tilde{F}(T)$, if
\begin{equation*}
D_f(y,Tx)\leq D_f(y,x),\quad \forall x\in C,y\in \tilde{F}(T),
\end{equation*}
and for any bounded sequence $\{x_n\}\subseteq C$ with
\begin{equation*}
\displaystyle\lim_{n\rightarrow \infty}\big(D_f(y,x_n)-D_f(y,Tx_n)\big)=0,
\end{equation*}
it follows that
\begin{equation*}
\displaystyle\lim_{n\rightarrow \infty}D_f(Tx_n,x_n)=0.
\end{equation*}

A mapping $B:E\rightarrow2^{E^*}$ is called Bregman inverse strongly monotone on the set $C$, if $C\cap (int\,dom f)\neq \emptyset$ and for any $x,y\in C\cap (int\,dom f),\xi\in Bx$ and $\eta \in By$, we have that
\begin{equation*}
\big\langle \xi-\eta,\;\nabla f^*(\nabla f(x)-\xi)-\nabla f^*(\nabla f(y)-\eta)\big\rangle \geq0.
\end{equation*}

Let $B:E\rightarrow 2^{E^*}$ be a mapping. Then the mapping defined by
\begin{equation*}
B_\lambda ^f:=\nabla f^*\circ(\nabla f-\lambda B):E\rightarrow E
\end{equation*}
is called an anti-resolvent associated with $B$ and $\lambda$ for any $\lambda>0$.

Suppose that $A$ is a mapping of $E$ into $2^{E^{*}}$ for the real reflexive Banach space $E$. The effective domain of $A$ is denoted by $dom(A)$, that is, $dom(A)=\lbrace x\in E: Ax\neq \emptyset \rbrace$. A multi-valued mapping $A$ on $E$ is said to be monotone if $\langle x-y, u^{*}-v^{*}\rangle\geq 0 $ for all $ x,y \in dom(A), u^{*} \in Ax$ and $v^{*}\in Ay$. A monotone operator $A$ on $E$ is said to be maximal if graph $A$, the graph of $A$, is not a proper subset of the graph of any monotone operator on $E$.

Let $E$ be a real reflexive Banach space, $f:E\rightarrow (-\infty,+\infty]$ be a uniformly Fr$\acute{\text{e}}$chet differentiable and bounded on bounded subsets of $E$, then for any $\lambda >0$ the resolvent of $A$   defined by
\begin{equation*}
Res_A^f(x)=(\nabla f+\lambda A)^{-1}\circ \nabla f(x).
\end{equation*}
is a single-valued Bregman quasi-nonexpansive mapping from $E$ onto $dom(A)$ and $F(Res_A^f)=A^{-1}0$. We denote by $A_\lambda=\frac{1}{\lambda}(\nabla f-\nabla f(Res_A^f))$ the Yosida approximation of $A$ for any $\lambda>0$. We get from \cite[ prop 2.7, p.10]{p22} that
\begin{equation*}
A_\lambda(x)\in A\big(Res_A^f(x)\big),\:\forall x\in E,\;\lambda>0,
\end{equation*}
 (see \cite{p21}, too).
\begin{lem}\label{qctas}\cite{p23}
Let $E$ be a real reflexive Banach space and $f:E\rightarrow(-\infty,+\infty]$ be a Legendre function which is totally convex on bounded subsets of $E$. Also let $C$ be a nonempty closed and convex subset of $int\,dom f$ and $T:C\rightarrow 2^C$ be a multi valued Bregman quasi-nonexpansive mapping. Then the fixed point set $F(T)$ of $T$ is a closed and convex subset of $C$.
\end{lem}
\begin{lem}\label{ewca}\cite{p16}
Assume that $f:E\rightarrow \mathbb{R}$ is a Legendre function which is uniformly $Fr\acute{e}chet$ differentiable and bounded on bounded subsets of $E$. Let $C$ be a nonempty closed and convex subset of $E$. Also let $\{T_i:i=1, ..., N\}$ be $N$ Bregman strongly nonexpansive mapping which satisfy $\tilde{F}(T_i)=F(T_i)$ for each $1\leq i\leq N$ and let $T=T_NT_{N-1}...T_1$. If $F(T)$ and $\bigcap_{i=1}^NF(T_i)$ are nonempty, then $T$ is also Bregman strongly nonexpansive with $F(T)=\tilde{F}(T)$.
\end{lem}
\begin{lem}\label{vtye}\cite{p20}
Let $G:E\rightarrow 2^{E^*}$ be a maximal monotone operator and $B:E\rightarrow E^*$ be a Bregman inverse strongly monotone mapping such that $(G+B)^{-1}(0^*)\neq \emptyset$. Also let $f:E\rightarrow \mathbb{R}$ be a Legendre function which is uniformly $Fr\acute{e}chet$ differentiable and bounded on bounded subset of $E$. Then
\begin{itemize}
  \item [{\rm(i)}] $(G+B)^{-1}(0^*)=F(Res_{\lambda G}^f\circ B_\lambda ^f)$.
  \item [{\rm(ii)}]  $Res_{\lambda G}^f\circ B_\lambda ^f$ is a Bregman strongly nonexpansive mapping such that
  \begin{equation*}
F(Res_{\lambda G}^f\circ B_\lambda ^f)=\tilde{F}(Res_{\lambda G}^f\circ B_\lambda ^f).
  \end{equation*}
  \item [{\rm(iii)}] $ D_f\big(u,Res_{\lambda G}^f\circ B_\lambda ^f(x)\big)+D_f\big(Res_{\lambda G}^f\circ B_\lambda ^f(x),x\big)\leq D_f(u,x),\;\forall u\in(G+B)^{-1}(0^*),\;x\in E\;and\;\lambda>0$.
  \end{itemize}
\end{lem}
\begin{lem}\label{awqqq}\cite{p26}
Let $f:E\rightarrow (-\infty,+\infty]$ be a proper convex and lower semicontinuous Legendre function. Then for any $z\in E$, for any $\{x_n\}\subseteq E$ and $\{t_i\}_{i=1}^N\subseteq (0,1)$ with $\sum_{i=1}^Nt_i=1$, the following holds
\begin{equation}\label{admwr}
D_f\bigg(z,\nabla f^*\big(\sum_{i=1}^Nt_i\nabla f(x_i)\big)\bigg)\leq \sum_{i=1}^Nt_iD_f(z,x_i).
\end{equation}
\end{lem}
\begin{pr}\label{awlzx}(\cite{p22} prop.2.8,p10)
Let $f$ be a $G\hat{a}teaux$ differentiable and $A:E\rightarrow 2^{E^*}$ be a maximal monotone operator such that $A^{-1}0\neq \emptyset$. Then
\begin{equation*}
D_f(q,x)\geq D_f\big(q,Res_{rA}^f(x)\big)+D_f\big(Res_{rA}^f(x),x\big).
\end{equation*}
for all $r>0,\;q\in A^{-1}0$ and $x\in E$.
\end{pr}
Next, we generalize the $k$-demimetric notation introduced in \cite{p31}.
\begin{de}\label{ertyu}
Let $E$ be a reflexive Banach space, $f:E\rightarrow (-\infty,+\infty]$ a Legendre function which is $G\hat{a}teaux$ differentiable and suppose $C$ be a nonempty, closed and convex subset of $int\,dom\,f$ and let $k\in(-\infty,1)$. A mapping $T:C\rightarrow int\,dom\,f$ with $F(T)\neq \emptyset$ is said Bregman $k$-demimetric, if for $x\in C$ and $q\in F(T)$,
\begin{equation*}
\big\langle x-q,\nabla f(x)-\nabla fT(x)\big\rangle\geq (1-k)D_f\big(x,T(x)\big).
\end{equation*}
\end{de}
\begin{ex}
Every Bregman quasi-nonexpansive mapping with the required conditions in Definition \ref{ertyu} is a Bregman $0$-demimetric mapping.
Let $p\in F(T)\neq\emptyset$ and $x\in C$, and we have
\begin{align*}
D_f(p,Tx)\leq D_f(p,x)\Rightarrow 0&\leq D_f(p,x)-D_f(p,Tx)\\
&=D_f(p,x)+D_f(x,Tx)-D_f(p,Tx)-D_f(x,Tx)\\
&=\big\langle p-x,\nabla fT(x)-\nabla f(x)\big\rangle-D_f(x,Tx),
\end{align*}
therefore
\begin{align*}
\big\langle x-p,\nabla f(x)-\nabla fT(x)\big\rangle\geq D_f(x,Tx).
\end{align*}
\end{ex}
\begin{ex}
From \cite[Lemma 2.1]{p30}, every Bregman quasi-strictly pseudo contractive mapping with the required conditions is a Bregman $k$-demimetric mapping for $k\in[0,1)$.
\end{ex}
\begin{ex}
Let $E$ be a reflexive Banach space, $f:E\rightarrow (-\infty,+\infty]$ a Legendre function which is $G\hat{a}teaux$ differentiable and $A:E\rightarrow 2^{E^*}$ a maximal monotone operator with $A^{-1}0\neq \emptyset$ and $r>0$. Then the $f$-resolvent $Res_{rA}^f$ is Bregman $0$-demimetric. In fact, from \eqref{awpon} and Proposition \ref{awlzx}, we have that
\begin{align*}
D_f(q,x)&\geq D_f\big(q,Res_{rA}^f(x)\big)+D_f\big(Res_{rA}^f(x),x\big)\\
&=D_f(q,x)+\big\langle q-Res_{rA}^f(x),\nabla f(x)-\nabla f(Res_{rA}^f(x))\big\rangle,
\end{align*}
for any $x\in E$ and $q\in A^{-1}0$. Then we obtain,
\begin{align*}
 \big\langle Res_{rA}^f(x)-q,\nabla f(x)-\nabla f(Res_{rA}^f(x))\big\rangle \geq 0,
\end{align*}
therefore
\begin{align*}
 \big\langle Res_{rA}^f(x)-x+x-q,\nabla f(x)-\nabla f(Res_{rA}^f(x))\big\rangle \geq 0,
\end{align*}
and hence from \eqref{awpon}, we have that
\begin{align}
 \big\langle x-q,\nabla f(x)-\nabla f(Res_{rA}^f(x))\big\rangle &\geq\big\langle x-Res_{rA}^f(x),\nabla f(x)-\nabla f(Res_{rA}^f(x))\big\rangle\nonumber \\
 &=D_f\big(x, Res_{rA}^f(x)\big)+D_f\big(Res_{rA}^f(x),x\big)\nonumber,
\end{align}
thus
\begin{align*}
\big\langle x-q,\nabla f(x)-\nabla f(Res_{rA}^f(x))\big\rangle\geq D_f\big(x,Res_{rA}^f(x)\big),
\end{align*}
and then we get that $Res_{rA}^f$ is Bregman $0$-demimetric.
\end{ex}
\section{\textbf{Main results}}
The following lemma is important and crucial in the proof of Theorem \ref{asli}.
\begin{lem}\label{meas}
  Let $E$ be a reflexive Banach space, $f:E\rightarrow (-\infty,+\infty]$ a Legendre function which is $G\hat{a}teaux$ differentiable and suppose $C$ be a nonempty, closed and convex subset of $int\,dom\,f$ and let $k$ be a real number with $k\in(-\infty,1)$ and let $T$ be a Bregman $k$-demimetric mapping of $C$ into $int\,dom\,f$. Then $F(T)$ is closed and convex.
\end{lem}
\begin{proof}
   First we show that $F(T)$ is closed. Consider a sequence $\{q_n\}$ such that $q_n\rightarrow q$ and $q_n\in F(T)$. We conclude from the definition of $T$ that
  \begin{align*}
  \big\langle q-q_n,\nabla f(q)-\nabla fT(q)\big\rangle\geq(1-k)D_f\big(q,T(q)\big).
  \end{align*}
   Since $q_n\rightarrow q$ we have $0\geq(1-k)D_f(q,T(q))$. Then from $1-k>0$, we have that $D_f(q,T(q))=0$ and hence $q=T(q)$ and therefore $q\in F(T)$. This implies that $F(T)$ is closed.

   Next, we show that $F(T)$ is convex. Suppose $p,q\in F(T)$ and set $z=\alpha p+(1-\alpha)q$, where $\alpha\in [0,1]$. Then we have that
  \begin{align*}
  \big\langle z-p,\nabla f(z)-\nabla fT(z)\big\rangle\geq(1-k)D_f\big(z,T(z)\big),\\
  \big\langle z-q,\nabla f(z)-\nabla fT(z)\big\rangle\geq(1-k)D_f\big(z,T(z)\big).
  \end{align*}
  Thus from $\alpha\geq0$ and $1-\alpha\geq0$, we have that
  \begin{equation*}
  \big\langle \alpha z-\alpha p,\nabla f(z)-\nabla fT(z)\big\rangle\geq\alpha(1-k)D_f\big(z,T(z)\big),
  \end{equation*}
  \begin{equation*}
  \big\langle(1-\alpha)z-(1-\alpha)q,\nabla f(z)-\nabla fT(z)\big\rangle\geq(1-\alpha)(1-k)D_f\big(z,T(z)\big).
  \end{equation*}
  From these inaqualities, we get that
  \begin{equation*}
  0=\big\langle z-z,\nabla f(z)-\nabla fT(z)\big\rangle\geq(1-k)D_f\big(z,T(z)\big).
  \end{equation*}
  Thus, we get that $D_f\big(z,T(z)\big)=0$ hence $z=Tz$ therefore $z\in F(T)$. We  conclude that $F(T)$ is convex.
 \end{proof}

 \begin{thm}\label{asli}
Let $E$ be a real reflexive Banach space. Suppose $f:E\rightarrow \mathbb{R}$ is a proper, convex, lower semicontinuous, strongly coercive, Legendre function which is bounded on bounded subsets of $E$, uniformly Fr\'{e}chet differentiable and totally convex on bounded subsets of $E$. Let $C$ be a nonempty, closed and convex subset of $int\,domf$. Let $\{k_1, k_2, ..., k_M\}\subseteq (-\infty, 1)$ and $\{T_j\}_{j=1}^M$ be a finite family of Bregman $k_j$-demimetric and Bregman quasi-nonexpansive and demiclosed mappings of $C$ into itself. Suppose $\{B_i\}_{i=1}^N$ is a finite family of Bregman inverse strongly monotone mappings of $C$ into $E$ and $\{B_{i,{\eta_n}}^f\}_{i=1}^N$ is the family of anti resolvent mappings of $\{B_i\}_{i=1}^N$. Let $A:E\rightarrow 2^{E^*}$ and $G:E\rightarrow 2^{E^*}$ be maximal monotone mappings on $E$ and let $ Q_\eta=Res_{\eta G}^f = (\nabla f+\eta G)^{-1}\nabla f$ and $J_r=Res_{r A}^f = (\nabla f+rA)^{-1}\nabla f$ be the resolvents of $G$ and $A$ for $\eta>0$ and $r>0$, respectively. Assume that
  \begin{equation*}
  \Omega =A^{-1}0\cap \big(\cap_{j=1}^MF(T_j)\big)\cap \big(\cap_{i=1}^N(B_i+G)^{-1}0^*\big)\neq \emptyset.
  \end{equation*}
For $x_1\in C$ and $C_1=Q_1=C$, let $\{x_n\}$ be a sequence defined by
  \begin{equation}\label{mnzx}
  \begin{cases}
  y_n=\nabla f^{*}\bigg(\sum_{j=1}^{M}\xi_j\big((1-\lambda_n)\nabla f+\lambda_n\nabla f T_j\big)x_n\bigg),\\
  z_n=\nabla f^{*}\sum_{i=1}^{N}\sigma_i\nabla f Q_{\eta_n}B_{i,{\eta_n}}^f(y_n),\\
  u_n= J_{r_n}z_n,\\
  C_{n+1}=\big\lbrace z\in C_n: D_f(z,y_n)\leq D_f(z,x_n), D_f(z,z_n)\leq D_f(z,y_n)\\
  \qquad\qquad\qquad\qquad\qquad and\: \langle z_n-z,\nabla f(z_n)-\nabla f(u_n)\rangle \geq D_f(z_n,u_n)\big\rbrace,\\
  x_{n+1}=Proj_{C_{n+1}\cap Q_n\emph{}}^f(x_1),\quad \forall n\in \mathbb{N},\\
  Q_{n+1}=\big\lbrace z\in Q_n:\langle x_{n+1}-z,\nabla f(x_1)-\nabla f(x_{n+1})\rangle\geq 0\big\rbrace,
  \end{cases}
  \end{equation}
where $\{\lambda_n\}\subseteq (0,1)$, $\{\eta_n\}$, $\{r_n\}\subseteq(0,+\infty)$, $\{\xi_1, \xi_2, ..., \xi_M\}$, $\{\sigma_1, \sigma_2, ...,\sigma_N\}\subseteq (0,1)$ and $a,b,c\in \mathbb{R}$ satisfy the following:
  \begin{itemize}
  \item [{\rm(1)}] $0<a\leq\lambda_n\quad \forall n\in \mathbb{N}$,
  \item [{\rm(2)}] $0<c\leq r_n, \quad \forall n\in \mathbb{N}$,
  \item [{\rm(3)}] $\sum_{j=1}^M\xi_j=1 \quad and \quad \sum_{i=1}^{N}\sigma_i=1$.
  \end{itemize}

  Then $\{x_n\}$ converges strongly to a point $\omega_0\in \Omega$ where $\omega_0=Proj_\Omega^fx_1$.
  \end{thm}
  \begin{proof}
   We divide the proof into several steps:

   $Step\,1$: First we prove that $\Omega$ is a closed and convex subset of $C$.\\
    Since $\{T_j\}_{j=1}^M$ is a finite family of $k_j$-Bregman demimetric mappings, by Lemma \ref{meas} and the condition $\Omega\neq\emptyset$, $F(T_j)$ is nonempty, closed and convex for $1\leq j\leq M$. Also, it follows from Lemma \ref{vtye} (i)-(ii) that $(B_i+G)^{-1}0^*=F(Q_{\eta_n}B_{i,\eta_n}^f)$ and $Q_{\eta_n}B_{i,\eta_n}^f$ is a Bregman strongly nonexpansive mapping and therefore from Lemma \ref{ewca} we have that
    \begin{equation*}
      F(Q_{\eta_n}B_{i,\eta_n}^f)=\tilde{F}(Q_{\eta_n}B_{i,\eta_n}^f).
    \end{equation*}
    Thus $\{Q_{\eta_n} B_{i,\eta_n}^f\}$ is a family of Bregman quasi-nonexpansive mappings. Using $\Omega\neq \emptyset$ and Lemma \ref{qctas}, we see that $(B_i+G)^{-1}0^*=F(Q_{\eta_n} B_{i,\eta_n}^f)$ is a nonempty, closed and convex set. We also know that $A^{-1}0$ is closed and convex. Then, $\Omega$ is nonempty, closed and convex. Therefore $Proj_\Omega^f$ is well defined.

    $Step\,2$: We prove that $C_n$ and $Q_n$ are closed and convex subsets of $C$ and $\Omega\subseteq C_{n+1}\cap Q_n,\forall n\in \mathbb{N}$.\\
    In fact, it is clear that $C_1=C$ is closed and convex. Suppose that $C_k$ is closed and convex for some $k\geq1$. Note that
    \begin{align*}
    D_f(z,y_k)\leq D_f(z,x_k)&\Leftrightarrow f(z)-f(y_k)-\langle\nabla f(y_k),z-y_k\rangle\leq f(z)-f(x_k)-\langle\nabla f(x_k), z-x_k\rangle\\
    &\Leftrightarrow\langle\nabla f(x_k),z-x_k\rangle-\langle\nabla f(y_k),z-y_k\rangle\leq f(y_k)-f(x_k).
    \end{align*}
    Similarly, we have that
    \begin{align*}
    D_f(z,z_k)\leq D_f(z,y_k)\Leftrightarrow\langle\nabla f(y_k),z-y_k\rangle-\langle\nabla f(z_k),z-z_k\rangle\leq f(z_k)-f(y_k).
    \end{align*}
    Thus from the fact that $D_f(.,x)$ is continuous for each fixed $x$  and using the above inequalities $\{z\in C_k: D_f(z,y_k)\leq D_f(z,x_k)\}$ and $\{z\in C_k: D_f(z,z_k)\leq D_f(z,y_k)\}$ are closed and convex. We have also $\{z\in C_k: \langle z_k-z,\nabla f(z_k)-\nabla f(u_k)\rangle \geq D_f(z_k,u_k)\}$ is  closed and convex. Therefore, $C_{k+1}$ is closed and convex. From mathematical induction we have that $C_n$ is a closed and convex subset in $C$ with $C_{n+1}\subseteq C_n$ for all $n\in \mathbb{N}$.\\
   Also, it is clear that $Q_1=C$ is closed and convex. Suppose that $Q_k$ is closed and convex for some $k\geq1$. Hence $\{z\in Q_k$: $\langle x_{k+1}-z,\nabla f(x_1)-\nabla f(x_{k+1})\rangle\geq0\}$ is closed and convex, i.e., $Q_{k+1}$ is a closed and convex subset of $Q_k$. Therefore by  mathematical induction $Q_n$ is a closed and convex subset of $C$ with $Q_{n+1}\subseteq Q_n$ for all $n\in \mathbb{N}$. Next, we show that $\Omega \subseteq C_n$ for all $n\geq 1$. Clearly $\Omega \subseteq C_1=C$. Assume that  $\Omega \subseteq C_k$ for some $k\in \mathbb{N}$. Note from Lemma \ref{awqqq} that
  \begin{align}\label{mner}
    D_f(z,y_k)&= D_f\big(z,\nabla f^{*}\sum_{j=1}^{M}\xi_j((1-\lambda_k)\nabla f+\lambda_k\nabla f T_j)x_k\big)\nonumber\\
    &=D_f\big(z,\nabla f^{*}\sum_{j=1}^{M}\xi_j\nabla f\nabla f^*((1-\lambda_k)\nabla f+\lambda_k\nabla f T_j)x_k\big)\nonumber\\
    &\leq \sum_{j=1}^{M}\xi_j D_f\big(z,\nabla f^*((1-\lambda_k)\nabla f+\lambda_k\nabla f T_j)x_k\big)\\
    &\leq \sum_{j=1}^{M}\xi_j\big((1-\lambda_k)D_f(z,x_k)+\lambda_kD_f(z,T_jx_k)\big)\nonumber\\
    &\leq \sum_{j=1}^{M}\xi_j\big((1-\lambda_k)D_f(z,x_k)+\lambda_kD_f(z,x_k)\big)\nonumber\\
    &=D_f(z,x_k),\nonumber
  \end{align}
  for all $z\in \Omega$. Furthermore, since $B_i$ is a Bregman inverse strongly monotone mapping for all $1\leq i\leq N$ and hence from Lemmas \ref{vtye} and \ref{awqqq} we have that
  \begin{align}\label{mnyu}
    D_f(z,z_k)&=D_f(z,\nabla f^{*}\sum_{i=1}^{N}\sigma_i\nabla f Q_{\eta_k}B_{i,{\eta_k}}^f y_k)\nonumber\\
    &\leq \sum_{i=1}^{N}\sigma_iD_f(z,Q_{\eta_k}B_{i,{\eta_k}}^f y_k)\leq \sum_{i=1}^{N}\sigma_iD_f(z,y_k)\\
    &=D_f(z,y_k),\nonumber
  \end{align}
  for all $z\in \Omega$.
  Also, since $J_{r_k}$ is the resolvent of $A$ and $u_k=J_{r_k}z_k$, we have from \eqref{awpon} and Proposition \ref{awlzx} that
  \begin{align}\label{mnyo}
  \langle z_k-z,\nabla f(z_k)-\nabla f(u_k)\rangle&=D_f(z,z_k)+D_f(z_k,u_k)-D_f(z,u_k)\nonumber\\
  &=D_f(z,z_k)+D_f(z_k,u_k)-D_f(z,J_{r_k}z_k)\\
  &\geq D_f(z,z_k)+D_f(z_k,u_k)-D_f(z,z_k)=D_f(z_k,u_k)\nonumber,
  \end{align}
  for all $z\in \Omega$. From \eqref{mner},\eqref{mnyu} and \eqref{mnyo}, we have that $\Omega \subseteq C_{k+1}$. Therefore, we have by mathematical induction that $\Omega \subseteq C_n$ for all $n\in \mathbb{N}$.\\
  Now, we shall show that $\Omega \subseteq Q_n$ for all $n\in \mathbb{N}$. Note that $\Omega \subseteq Q_1=C$. Assume that $\Omega \subseteq Q_k$ for some $k\in \mathbb{N}$. Thus, $\Omega \subseteq C_{k+1}\cap Q_k$ for some $k\in \mathbb{N}$. From $x_{k+1}=Proj_{C_{k+1}\cap Q_k}^f(x_1)$ and Lemma \ref{rtmnw}, we have that
  \begin{align*}\label{mnkl}
  \langle x_{k+1}-z,\nabla f(x_1)-\nabla f(x_{k+1})\rangle\geq0,\qquad \forall z\in C_{k+1}\cap Q_k.
  \end{align*}
  Since $\Omega \subseteq C_{k+1}\cap Q_k$, we have that
  \begin{align*}
  \langle x_{k+1}-z,\nabla f(x_1)-\nabla f(x_{k+1})\rangle\geq 0,\qquad \forall z\in \Omega.
  \end{align*}
  Then, we get $\Omega \subseteq Q_{k+1}$. By mathematical induction, we have that $\Omega \subseteq Q_n$ for all $n\in \mathbb{N}$. This implies that $\{x_n\}$ is well defined.

  $Step\,3$: We show that $\displaystyle\lim_{n\rightarrow \infty} D_f(x_n,x_1)$ exists.\\
  Since $\Omega$ is nonempty, closed and convex, there exists a $\omega_0\in \Omega$ such that $\omega_0=Proj_{\Omega}^f(x_1)$. From $x_{n+1}=Porj_{C_{n+1}\cap Q_n}^f(x_1)$, we get that
  \begin{equation*}
    D_f(x_{n+1},x_1)\leq D_f(z,x_1),
  \end{equation*}
  for all $z\in C_{n+1}\cap Q_n$. From $\omega_0\in \Omega \subseteq C_{n+1}\cap Q_n$, we obtain that
  \begin{equation}\label{bher}
  D_f(x_{n+1},x_1)\leq D_f(\omega_0,x_1).
  \end{equation}
  This shows that $\{D_f(x_n,x_1)\}$ is a bounded sequance. By Lemma \ref{rtmnw} (iii), we have that
  \begin{align}\label{jkwe}
  0&\leq D_f(x_{n+1},x_n)=D_f\big(x_{n+1},Proj_{C_n\cap Q_{n-1}}^f(x_1)\big)\nonumber\\
  &\leq D_f(x_{n+1},x_1)-D_f\big(Proj_{C_n\cap Q_{n-1}}^f(x_1),x_1\big)\\
  &=D_f(x_{n+1},x_1)-D_f(x_n,x_1),\nonumber
  \end{align}
  for all $n>1$. Therefore
  \begin{equation}\label{fgrt}
  D_f(x_n,x_1)\leq D_f(x_{n+1},x_1).
  \end{equation}
  This implies that $\{D_f(x_n,x_1)\}$ is bounded and nondecreasing. Then  $\displaystyle\lim_{n\rightarrow \infty}D_f(x_n,x_1)$ exists. In view of Lemma \ref{qwpo}, we deduce that the sequence $\{x_n\}$ is bounded. Also, from \eqref{jkwe}, we have that
  \begin{equation}\label{dfpo}
  \lim_{n\rightarrow \infty} D_f(x_{n+1},x_n)=0.
  \end{equation}
  Since the function $f$ is totally convex on bounded sets, by Lemma \ref{qpocv}, we have that
  \begin{equation}\label{asol}
  \lim_{n\rightarrow \infty}\|x_{n+1}-x_n\|=0.
  \end{equation}

  $Step\,4$: We prove that $\{x_n\}$ is a Cauchy sequence in $C$.\\
  We have $C_m\subseteq C_n$ and $Q_m\subseteq Q_n$ for any $m,n\in \mathbb{N}$ with $m\geq n$. From $x_n=Proj_{C_n\cap Q_{n-1}}^f(x_1)\in C_n\cap Q_{n-1}$ and Lemma \ref{rtmnw} we have that
  \begin{align}\label{asno}
  D_f(x_m,x_n)&=D_f\big(x_m,Proj_{C_n\cap Q_{n-1}}^f(x_1)\big)\nonumber\\
  &\leq D_f(x_m,x_1)-D_f\big(Proj_{C_n\cap Q_{n-1}}^f(x_1),x_1\big)\\
  &=D_f(x_m,x_1)-D_f(x_n,x_1).\nonumber
  \end{align}
  Letting $m,n\rightarrow \infty$ in \eqref{asno}, we deduce that $D_f(x_m,x_n)\rightarrow 0$. In view of Lemma \ref{qpocv}, since the function $f$ is totally convex on bounded sets, we get that $\|x_m-x_n\|\rightarrow0$ as $m,n\rightarrow \infty$. Thus $\{x_n\}$ is a Cauchy sequence. Since $E$  is a Banach space and $C$ is closed and convex, we conclude that there exists $\bar{x}\in C$ such that
  \begin{align}\label{awlo}
   \displaystyle\lim_{n\rightarrow \infty}\|x_n-\bar{x}\|=0.
  \end{align}

  $Step\,5$: We prove that $\displaystyle\lim_{n\rightarrow \infty}\|\nabla f(x_n)-\nabla f(y_n)\|=0$.\\
  Using \eqref{dfpo} and from $x_{n+1}\in C_{n+1}$, we get that
  \begin{align}\label{azlk}
  D_f(x_{n+1},y_n)\leq D_f(x_{n+1},x_n)\rightarrow 0\,\;(as\;\;n\rightarrow \infty),
  \end{align}
  then $\lim_{n\rightarrow \infty}D_f(x_{n+1},y_n)=0$. Since the function $f$ is totally convex on bounded sets and $\{y_n\}$ is bounded, by Lemma \ref{qpocv}, we have that
  \begin{equation}\label{asqv}
  \lim_{n\rightarrow \infty}\|x_{n+1}-y_n\|=0.
  \end{equation}
  Using \eqref{asol} and \eqref{asqv}, we have that
  \begin{align*}
  \|x_n-y_n\|\leq \|x_n-x_{n+1}\|+\|x_{n+1}-y_n\|\rightarrow 0\,\;(as\;\;n\rightarrow \infty),
  \end{align*}
  then
  \begin{align}\label{astp}
  \lim_{n\rightarrow \infty}\|x_n-y_n\|=0.
  \end{align}
  Since from Lemma \ref{zakhv}, $\nabla f$ is uniformly continuous, we have that
  \begin{equation}\label{awmx}
  \lim_{n\rightarrow \infty}\|\nabla f(x_n)-\nabla f(y_n)\|=0.
  \end{equation}

 $Step\,6$: We prove that $\displaystyle\lim_{n\rightarrow \infty}\|\nabla f(y_n)-\nabla f(z_n)\|=0$.\\
  Using \eqref{azlk} and from $x_{n+1}\in C_{n+1}$, we get that
  \begin{align}\label{aedm}
  D_f(x_{n+1},z_n)\leq D_f(x_{n+1},y_n)\rightarrow 0\,\;(as\;\;n\rightarrow \infty),
  \end{align}
   then $\lim_{n\rightarrow \infty}D_f(x_{n+1},z_n)=0$. Since function $f$ is totally convex on bounded sets and $\{z_n\}$ is bounded, by Lemma \ref{qpocv}, we have that
  \begin{align}\label{aqhn}
  \lim_{n\rightarrow \infty}\|x_{n+1}-z_n\|=0.
  \end{align}
   Using \eqref{asol} and \eqref{aqhn}, we have that
  \begin{align}\label{aqhnui}
  \lim_{n\rightarrow \infty}\|x_n-z_n\|=0.
  \end{align}
  Applying \eqref{astp} and \eqref{aqhnui}, we get that
  \begin{align}\label{aqekgi}
  \lim_{n\rightarrow \infty}\|y_n-z_n\|=0.
  \end{align}
  Since $\nabla f$ is uniformly continuous, we have that
  \begin{align}\label{awlfm}
  \lim_{n\rightarrow \infty}\|\nabla f(y_n)-\nabla f(z_n)\|=0.
  \end{align}

  $Step\,7$: We prove that $\displaystyle\lim_{n\rightarrow \infty}\|\nabla f(z_n)-\nabla f(u_n)\|=0$.\\
  Using $x_{n+1}\in C_{n+1}$ and from \eqref{awpon} we get that
  \begin{align*}
  D_f(z_n,u_n)&\leq \langle z_n-x_{n+1},\nabla f(z_n)-\nabla f(u_n)\rangle\\
  &=D_f(x_{n+1},z_n)+D_f(z_n,u_n)-D_f(x_{n+1},u_n),
  \end{align*}
  then
  \begin{align*}
  0\leq D_f(x_{n+1},z_n)-D_f(x_{n+1},u_n),
  \end{align*}
  therefore from \eqref{aedm}, we have that
  \begin{align}\label{aewpo}
  D_f(x_{n+1},u_n)\leq D_f(x_{n+1},z_n)\rightarrow 0\,\;(as\;\;n\rightarrow \infty),
  \end{align}
  and then $\lim_{n\rightarrow \infty}D_f(x_{n+1},u_n)=0$. Since the function $f$ is totally convex on bounded sets and from Proposition \ref{awlzx} and Lemma \ref{fdyuq}, $\{u_n\}$ is bounded hence by Lemma \ref{qpocv}, we have that
  \begin{align}\label{awjmn}
  \lim_{n\rightarrow \infty}\|x_{n+1}-u_n\|=0.
  \end{align}
   Using \eqref{asol} and \eqref{awjmn}, we conclude that
  \begin{align}\label{astpw}
  \lim_{n\rightarrow \infty}\|x_n-u_n\|=0.
  \end{align}
  Now, by \eqref{aqhnui} and \eqref{astpw} we get that
  \begin{align}\label{argyj}
  \lim_{n\rightarrow \infty}\|z_n-u_n\|=0.
  \end{align}
   Since from Lemma \ref{zakhv}, $\nabla f$ is uniformly continuous, we have that
  \begin{align}\label{awvbe}
  \lim_{n\rightarrow \infty}\|\nabla f(z_n)-\nabla f(u_n)\|=0.
  \end{align}

  $Step\,8$: We prove that $\bar{x}\in \Omega$.\\
  Since $T_j$ is Bregman $k_j$-demimetric for all $1\leq j\leq M$, we get that
  \begin{align*}
  \langle x_n-z,\nabla f(x_n)-\nabla f(y_n)\rangle&=\big\langle x_n-z,\nabla f(x_n)-\nabla f\nabla f^{*}\bigg(\sum_{j=1}^{M}\xi_j\big((1-\lambda_n)\nabla f+\lambda_n\nabla f T_j\big)x_n\bigg)\big\rangle\\
  &=\sum_{j=1}^{M}\xi_j\big\langle x_n-z,\nabla f(x_n)-\big((1-\lambda_n)\nabla f+\lambda_n\nabla f T_j\big)x_n\big\rangle\\
  &=\sum_{j=1}^{M}\xi_j\lambda_n\langle x_n-z,\nabla f(x_n)-\nabla fT_j(x_n)\rangle\\
  &\geq \sum_{j=1}^{M}\xi_j\lambda_n(1-k)D_f(x_n,T_jx_n)\\
  &\geq \sum_{j=1}^{M}\xi_ja(1-k)D_f(x_n,T_jx_n),
  \end{align*}
  for all $z\in \cap_{j=1}^MF(T_j)$. We have from \eqref{awmx} that
  \begin{align*}
    \displaystyle\lim_{n\rightarrow\infty}D_f(x_n,T_jx_n)=0,
  \end{align*}
   for all $1\leq j\leq M$. Since the function $f$ is totally convex on bounded sets, by Lemma \ref{qpocv}, we have that
  \begin{align}\label{aqoxg}
  \lim_{n\rightarrow \infty}\|x_n-T_jx_n\|=0,\quad \forall j\in\{1,...,M\}.
  \end{align}
   Since $T_j$ is demiclosed for all $1\leq j\leq M$ and from \eqref{awlo}, $x_n\rightarrow\bar{x}$ as $n\rightarrow\infty$, we have that $\bar{x}\in\cap_{j=1}^MF(T_j)$. We now show that $\bar{x}\in\cap_{i=1}^N(B_i+G)^{-1}0^*$. From \eqref{awpon} and Lemma \ref{vtye} (iii), we get that
  \begin{align*}
  \langle y_n-z,\nabla f(y_n)-\nabla f(z_n)\rangle&=\big\langle y_n-z,\nabla f(y_n)-\nabla f\nabla f^{*}\sum_{i=1}^{N}\sigma_i\nabla f Q_{\eta_n}B_{i,{\eta_n}}^f(y_n)\big\rangle\\
  &=\sum_{i=1}^{N}\sigma_i\big\langle y_n-z,\nabla f(y_n)-\nabla f Q_{\eta_n}B_{i,{\eta_n}}^f(y_n)\big\rangle\\
  &=\sum_{i=1}^{N}\sigma_i\big(D_f(z,y_n)+D_f(y_n,Q_{\eta_n}B_{i,{\eta_n}}^fy_n)-
  D_f(z,Q_{\eta_n}B_{i,{\eta_n}}^fy_n)\big)\\
  &\geq \sum_{i=1}^{N}\sigma_i\big(D_f(z,y_n)+D_f(y_n,Q_{\eta_n}B_{i,{\eta_n}}^fy_n)-D_f(z,y_n)\big)\\
  &=\sum_{i=1}^{N}\sigma_i\big(D_f(y_n,Q_{\eta_n}B_{i,{\eta_n}}^fy_n)\big),
  \end{align*}
  for all $z\in \cap_{i=1}^N(B_i+G)^{-1}0^*$ and $i\in \{1,...,N\}$. Using \eqref{awlfm}, from the above, we have that
  \begin{align*}
  \displaystyle\lim_{n\rightarrow \infty}D_f(y_n,Q_{\eta_n}B_{i,{\eta_n}}^fy_n)=0,
  \end{align*}
  for all $1\leq i\leq N$. Since the function $f$ is totally convex on bounded sets, by Lemma \ref{qpocv}, we conclude that
  \begin{align}\label{azjku}
  \displaystyle\lim_{n\rightarrow \infty}\|y_n-Q_{\eta_n}B_{i,{\eta_n}}^fy_n\|=0.
  \end{align}
  On the other hand, it follows from \eqref{awlo} and \eqref{astp} that
  \begin{align}\label{aoeng}
  \displaystyle\lim_{n\rightarrow \infty}\|y_n-\bar{x}\|=0.
  \end{align}
  Now, from \eqref{azjku} and \eqref{aoeng}, we have that $\bar{x}\in\tilde{F}(Q_{\eta_n}B_{i,{\eta_n}})$ for all $i\in \{1,...,N\}$. From Lemma \ref{vtye}, we conclude that
  \begin{align*}
  \tilde{F}(Q_{\eta_n}B_{i,{\eta_n}})=F(Q_{\eta_n}B_{i,{\eta_n}})=(B_i+G)^{-1}0^*,\quad \forall i\in\{1,...,N\},
  \end{align*}
  and therefore $\bar{x}\in\cap_{i=1}^N(B_i+G)^{-1}0^*$.\\
  We now show $\bar{x}\in A^{-1}0$. Using $r_n\geq c$, we have from \eqref{awvbe} that
  \begin{align*}
  \lim_{n\rightarrow \infty}\frac{1}{r_n}\|\nabla f(z_n)-\nabla f(u_n)\|=0.
  \end{align*}
  So applying $A_{r_n}$, the Yosida approximation of $A$, we have
  \begin{align*}
  \displaystyle\lim_{n\rightarrow\infty}\|A_{r_n}z_n\|=\lim_{n\rightarrow \infty}\frac{1}{r_n}\|\nabla f(z_n)-\nabla f(u_n)\|=0.
  \end{align*}
  Since $A_{r_n}z_n \in Au_n$, for $(p,p^*)\in A$,  we have from the monotonicity of $A$ that $\langle p-u_n,p^*-A_{r_n}z_n\rangle\geq0$ for all $n\in \mathbb{N}$. By \eqref{awlo} and \eqref{aqhnui}, we have that
  \begin{align*}
  \displaystyle\lim_{n\rightarrow\infty}\|z_n-\bar{x}\|=0.
  \end{align*}
  From \eqref{argyj} and the above, we have $\|u_n-\bar{x}\|\rightarrow0$. Thus, we get
   $\langle p-\bar{x},p^*\rangle\geq0$. From the maximality of $A$, we have that $\bar{x}\in A^{-1}0$. Therefore $\bar{x}\in\Omega$.

   Since $\omega_0=Proj_{\Omega}^f(x_1),\;\bar{x}\in\Omega$ and $\|x_n-\bar{x}\|\rightarrow0$, we have from \eqref{bher} that
   \begin{align*}
   D_f(\omega_0,x_1)\leq D_f(\bar{x},x_1)=\displaystyle\lim_{n\rightarrow \infty}D_f(x_n,x_1)\leq \displaystyle\lim_{n\rightarrow \infty}D_f(\omega_0,x_1)=D_f(\omega_0,x_1),
   \end{align*}
   which implies that
   \begin{align*}
   \lim_{n\rightarrow \infty}D_f(x_n,x_1)=D_f(\omega_0,x_1),
   \end{align*}
   therefore,
  \begin{align}\label{qaswe}
   D_f(\bar{x},x_1)=D_f(\omega_0,x_1).
  \end{align}
   From \eqref{mnzx}, \eqref{awlo} and $\omega_0\in \Omega\subseteq Q_{n+1}$ for all $n\in\mathbb{N}$ and also since $\nabla f$ is uniformly continuous on bounded subsets and is bounded on bounded subsets, we have that
   \begin{align}\label{lkscv}
    \langle x_{n+1}-\omega_0,\nabla f(x_1)-\nabla f(x_{n+1})\rangle\geq 0 &\Rightarrow
    \lim_{n\rightarrow \infty}\langle x_{n+1}-\omega_0,\nabla f(x_1)-\nabla f(x_{n+1})\rangle\geq 0\nonumber\\
    &\Rightarrow \langle \bar{x}-\omega_0,\nabla f(x_1)-\nabla f(\bar{x})\rangle\geq 0.
   \end{align}
    Now, from \eqref{awpon} and \eqref{qaswe}, we get that
    \begin{align*}
      D_f(\omega_0,\bar{x})&= D_f(\omega_0,\bar{x})+D_f(\bar{x},x_1)-D_f(\omega_0,x_1)=\langle \omega_0-\bar{x},\nabla f(x_1)-\nabla f(\bar{x})\rangle.
    \end{align*}
    From \eqref{lkscv} and the above, we conclude that $D_f(\omega_0,\bar{x})\leq0$ and therefore $D_f(\omega_0,\bar{x})=0$. Thus $\omega_0=\bar{x}$ and hence $x_n\rightarrow \omega_0$. This completes the proof.
  \end{proof}

  Next, we prove a proposition to extend Theorem \ref{asli}.
  \begin{pr}\label{qbjs}
    Let $E$ be a real reflexive Banach space. Suppose $f:E\rightarrow \mathbb{R}$ be a proper, convex, lower semicontinuous, Legendre function and $G\hat{a}teaux$ differentiable of $E$. Let $C$ be a nonempty, closed and convex subset of $int\,domf$. Suppose $k\in(-\infty,0]$ and a mapping $T:C\rightarrow C$ with $F(T)\neq \emptyset$ be Bregman $k$-demimetric. Then $T$ is a Bregman quasi-nonexpansive mapping.
 \end{pr}
 \begin{proof}
 Let $p\in F(T)$ and $x\in C$. Using \eqref{awpon} and Definition \ref{ertyu}, we have
%
 \begin{align*}
 D_f(p,Tx)-D_f(p,x)&=D_f(x,Tx)-D_f(p,x)-D_f(x,Tx)+D_f(p,Tx)\\
 &=D_f(x,Tx)-\langle p-x,\nabla f(Tx)-\nabla f(x)\rangle\\
 &\leq D_f(x,Tx)-(1-k)D_f(x,Tx)\\
 &=kD_f(x,Tx)\leq0,
 \end{align*}
 then $T$ is a Bregman quasi-nonexpansive mapping.
 \end{proof}
 \begin{rem}
 From Proposition \ref{qbjs}, for each $1\leq j\leq M$, we may remove the condition that $T_j$ is Bregman quasi-nonexpansive in Theorem \ref{asli} when $k_j\leq0$.
 \end{rem}
 \begin{op}
Can one generalize Proposition \ref{qbjs} to Bregman k-demimetric mappings with $k<1$?
 \end{op}

\vspace{0.1in}
\hrule width \hsize \kern 1mm

\begin{thebibliography}{XX}
\bibitem{p20} Ogbuisi FU, Izuchukwu C. Approximating a zero of sum of two monotone operations which solves a fixed point problem in reflexive Banach spaces. Collectanea Mathematica (in print).
\end{thebibliography}
\end{document}